\documentclass[letterpaper,10pt]{amsart}
\usepackage{amsmath,amssymb,amsxtra,amsthm,hyperref,enumerate}

\makeatletter
\@namedef{subjclassname@2010}{%
  \textup{2010} Mathematics Subject Classification}
\makeatother

\newcommand{\bh}[1] {\mathcal{B}(\mathcal{#1})}
\newcommand{\diag} {\operatorname{diag}}

\newtheorem{theorem}{Theorem}[section]
\newtheorem{corollary}[theorem]{Corollary}
\newtheorem{lemma}[theorem]{Lemma}

\theoremstyle{definition}
\newtheorem{definition}[theorem]{Definition}

\newtheorem{example}[theorem]{Example}
\newtheorem{remark}[theorem]{Remark}

\begin{document}

\title{Normal Extensions of Representations of Abelian Semigroups}
\author{Boyu Li}
\address{Pure Mathematics Department\\University of Waterloo\\Waterloo, ON\\Canada \ N2L--3G1}
\email{b32li@math.uwaterloo.ca}
\date{\today}


\begin{abstract} A commuting family of subnormal operators need not have a commuting normal extension. We study when a representation on an abelian semigroup can be extended to a normal representation, and show that it suffices to extend the set of generators to commuting normals. We also extend a result due to Athavale to representations on abelian lattice ordered semigroups. 
\end{abstract}

\maketitle

\section{Introduction}
 
An operator $T\in\bh{H}$ is called subnormal if there exists a normal extension $N\in\bh{K}$ where $\mathcal{H}\subseteq\mathcal{K}$ and $N|_\mathcal{H}=T$. There are many equivalent conditions for an operator being subnormal, for example, Agler showed a contractive operator $T$ is subnormal if and only if for any $n\geq 0$, $$\sum_{j=0}^n (-1)^j {n\choose j} T^{*j}T^j\geq 0.$$
One may refer to \cite[Chapter II]{SubnormalBookConway} for many other characterizations of subnormal operators.

A commuting pair of subnormal operators $T_1, T_2\in\bh{H}$ might not have commuting normal extensions \cite{Lubin1978, Abrahamse1978}, and a necessary and sufficient condition was given by It{\^o} in \cite{Ito1958}. Athavale obtained a necessary and sufficient condition for $n$ commuting operators $T_1,\cdots,T_n\in\bh{H}$ to have commuting normal extensions in terms of operator polynomials \cite{A1987,AP1990}.

This paper consider the question as to when a contractive representation of a unital abelian semigroup can be extended to a contractive normal representation. Athavale's result can be applied to the set of generators, and obtain a map that sends the semigroup into a family of commuting normal operators. Our first result shows that such normal map guarantees the existence of a normal representation. 

It is also observed that Athavale's result is equivalent to a certain representation being regular, and we further extend Athavale's result to abelian lattice ordered semigroups. 

\section{Commuting Normal Extension}

For an operator $T\in\bh{H}$, an operator $S\in\bh{K}$ extends $T$ ($S$ is called an extension of $T$) if it acts on a larger Hilbert space $\mathcal{K}\supseteq\mathcal{H}$ such that $S|_\mathcal{H}=T$. In other words, $S$ has the form $$S=\begin{bmatrix} T & * \\ 0 & *\end{bmatrix}.$$
An operator $T\in\bh{H}$ is called subnormal if it has a normal extension. Among many characterizations of subnormal operators, Agler \cite{Agler1985} showed a contractive operator $T$ is subnormal if and only if for any $n\geq 0$, $$\sum_{j=0}^n (-1)^j {n\choose j} T^{*j}T^j\geq 0.$$
However, a commuting pair of subnormal operators need not have a commuting pair of normal extension \cite{Lubin1978, Abrahamse1978}. It{\^o} \cite{Ito1958} established a necessary and sufficient condition for a commuting family of subnormal operators to have commuting normal extensions. Athavale \cite{A1987} generalized Agler's result to a family of commuting contractions:
\begin{theorem}[Athavale]\label{thm.Athavale} Let $T=(T_1,T_2,\cdots,T_m)$ be a family of $m$ commuting contractions. Then $T$ has a commuting normal extension $N$ if and only if for any $n_1,n_2,\cdots,n_m\geq 0$, we have
\begin{equation}\label{eq.athavale}\sum_{0\leq k_i\leq n_i} (-1)^{k_1+k_2+\cdots+k_m} {n_1\choose k_1}\cdots{n_m\choose k_m} T_1^{*k_1}T_2^{*k_2}\cdots T_m^{*k_m}T_m^{k_m}\cdots T_1^{k_1}\geq 0.\tag{$\star$}\end{equation}
\end{theorem}
One may observe that a family of $m$ commuting contractions defines a contractive representation $T:\mathbb{N}^m\to\bh{H}$ that sends each generator $e_i$ to $T_i$. A commuting normal extension $N=(N_1,\cdots,N_m)$ can be seen as a contractive normal representation $N:\mathbb{N}^m\to\bh{K}$ that extends $T$. Athavale's result gives a necessary and sufficient condition for the existence of a normal representation that extends $T$. If $P$ is a unital abelian semigroup and $T:P\to\bh{H}$ is a contractive representation, we may also ask the question when there exists a normal representation $N:P\to\bh{K}$ that extends $T$. 
\begin{example}\label{ex.1} Consider $P=\mathbb{N}\backslash\{1\}$ which is a unital semigroup generated by $2$ and $3$. A contractive representation $T:P\to\bh{H}$ is uniquely determined by $T(2),T(3)$, which satisfies $T(2)^3=T(3)^2$. We may use Theorem \ref{thm.Athavale} to test if $T(2),T(3)$ has commuting normal extensions $N_2,N_3$. However, even if they do have such extensions, there is no guarantee that $N_2^3=N_3^2$ and therefore it is not clear if we can get a normal representation $N:P\to\bh{K}$ that extends $T$. Nevertheless, since $N_2,N_3$ extend $T(2),T(3)$ respectively, we may define a normal map $N:P\to\bh{K}$ using $N_2,N_3$ such that $\{N(p)\}_{p\in P}$ is a family of commuting normal operators where $N(p)$ extends $T(p)$. As we shall see soon, in Theorem \ref{thm.normalmap}, the existence of such normal map guarantees a normal representation that extend $T$.

We shall also note that this semigroup $P=\mathbb{N}\backslash\{1\}$ is closely related to the so-called Neil algebra $\mathcal{A}=\{f\in A(\mathbb{D}): f'(0)=0\}$. Dilation on Neil algebra has been studied in \cite{DJM2015, Broschinski2013}. Unlike $\mathbb{N}$ where every contractive representation has a unitary dilation due to Sz.Nagy's dilation, contractive representations of $P$ may not have a unitary dilation. Even so, for a contractive representation $T:P\to\bh{H}$, we may apply Ando's theorem to dilate $T(2),T(3)$ into commuting unitaries $U_2,U_3$, and therefore there exists a family $\{U_n\}_{n\in P}$ of commuting unitaries where $P_\mathcal{H} U_n|_\mathcal{H} = T(n)$ for each $n$ \cite[Example 2.4]{DJM2015}. However, existence of such unitary maps does not guarantees a unitary dilation of $T$. 
\end{example}
One of the main tools for the proof is the involution semigroup. Sz.Nagy used such a technique and proved a subnormality condition of a single operator due to Halmos \cite{NagyFABookAppendix}, and Athavale also used this technique in \cite{A1987}. We shall extend this technique to a more general setting.
\begin{definition} A semigroup $P$ is called an involution semigroup (or a $*$-semigroup) if there is an involution $*:P\to P$ that satisfies $p^{**}=p$ and $(pq)^*=q^* p^*$. 
\end{definition}

For example, any group $G$ can be seen as an involution semigroup where $g^*=g^{-1}$. Any abelian semigroup can be seen as involution semigroup where $p^*=p$. A representation $D$ of a unital involution semigroup $P$ is a unital $*$-homomorphism. It is obvious that if $p p^*=p^* p$, then $D(p)$ is normal. Sz.Nagy established a condition which guarantees that a map on an involution semigroup has a dilation to a representation of the semigroup \cite{NagyFABookAppendix}.

\begin{theorem}\label{thm.principal} Let $P$ be a $*$-semigroup and $T:P\to\bh{H}$ satisfies the following conditions:
\begin{enumerate}[(i)]
\item $T(e)=I, T(p^*)=T(p)^*$,
\item For any $p_1,\cdots,p_n\in P$, the operator matrix $[T(p_i^* p_j)]$ is positive, 
\item There exists a constant $C_a>0$ for each $a\in P$ such that for all $p_1,\cdots,p_n\in P$, $$[T(p_i^* a^* a p_j)]\leq C_a^2 [T(p_i^* p_j)].$$
\end{enumerate}
Then, there exists a representation $D:P\to\bh{K}$ that satisfies $T(p)=P_\mathcal{H} D(p)|_\mathcal{H}$ and $\|D(p)\|\leq C_p$.
\end{theorem}

Now let $P$ be a unital abelian semigroup and consider $Q=\{(p,q):p,q\in P\}$. $Q$ is a unital semigroup under the point-wise semigroup operation $$(p_1,q_1)+(p_2,q_2)=(p_1+p_2,q_1+q_2).$$ 
Define a involution operation of $Q$ by $(p,q)^*=(q,p)$, which turns $Q$ into an involution semigroup. Notice since $P$ is abelian, $Q$ is also abelian. Moreover, any element $(p,q)=(0,q)+(0,p)^*$. If $D:Q\to\bh{K}$ is a representation, then $$D(0,p)^* D(0,p)=D(p,p)=D(0,p)D(0,p)^*,$$ and therefore $D(0,p)$ is normal. 

\begin{lemma}\label{lm.1} Let $T\in\bh{H}$ and $N\in\bh{K}$ where $\mathcal{H}$ is a subspace of $\mathcal{K}$. Suppose $T=P_\mathcal{H} N|_\mathcal{H}$ and $T^* T= P_\mathcal{H} N^*N|_\mathcal{H}$, then $N$ is an extension of $T$.
\end{lemma}
\begin{proof}
From the conditions, we have for any $h\in\mathcal{H}$, $\|Th\|^2=\langle Th, Th\rangle=\langle T^* Th,h\rangle$. Since $T^* T= P_\mathcal{H} N^*N|_\mathcal{H}$, $\langle T^* Th,h\rangle=\langle N^* Nh,h\rangle=\|Nh\|^2$. 

On the other hand, $\|Th\|=\sup_{\|k\|\leq 1, k\in\mathcal{H}} \langle Th, k\rangle$. But $T=P_\mathcal{H} N|_\mathcal{H}$, and thus $\langle Th,k\rangle=\langle Nh, k\rangle$. Therefore, 
\begin{align*}
\|Th\|& =\sup_{\|k\|\leq 1, k\in\mathcal{H}}\langle Th, k\rangle \\
&= \sup_{\|k\|\leq 1, k\in\mathcal{H}}\langle Nh, k\rangle \\
&= \|P_\mathcal{H} Nh \|
\end{align*}

Therefore, $\|Th\|=\| Nh\|=\|P_\mathcal{H} Nh\|$ and thus $\mathcal{H}$ is invariant for $N$. Hence, $N$ is an extension of $T$.\end{proof}
\begin{theorem}\label{thm.normalmap} Let $P$ be any unital abelian semigroup and let $T:P\to\bh{H}$ be a unital contractive representation of $P$. Then the following are equivalent:
\begin{enumerate}[(i)]
\item There exists a contractive normal map $N:P\to\bh{K}$ that extends $T$, where the family $\{N(p)\}_{p\in P}$ is a commuting family of normal operators.
\item There exists a contractive normal representation $N:P\to\bh{L}$ that extends $T$. 
\end{enumerate}
\end{theorem}

\begin{proof}

\textit{(ii)}$\Longrightarrow$\textit{(i)} is trivial. For the other direction, denote $Q$ be the $*$-semigroup constructed before and let $\tilde{T}:Q\to\bh{H}$ defined by $\tilde{T}(p,q)=T(p)^* T(q)$. For each $p\in P$, denote $N(p)=\begin{bmatrix} T(p) & X_p \\ 0 & Y_p \end{bmatrix}$. Pick $s_i=(p_i,q_i)\in Q$ and $t=(a,b)\in Q$. We shall show that $\tilde{T}$ satisfies all the conditions in Theorem \ref{thm.principal}. 

The first condition of Theorem \ref{thm.principal} is clearly valid. For the second condition:
\begin{align*}
 &[\tilde{T}(s_i^* s_j)] \\
=& [\tilde{T}(q_i p_j, p_i q_j)] \\
=& [T(q_i)^* T(p_j)^* T(p_i) T(q_j)] \\
=& \diag(T(q_1)^*, T(q_2)^*, \cdots, T(q_n)^*) [T(p_j)^* T(p_i)] \diag(T(q_1),\cdots,T(q_n))
\end{align*}
It suffices to show $[T(p_j)^* T(p_i)]\geq 0$. Notice that $\{N(p_i)\}$ is a commuting family of normal operators and thus they also doubly commute (by Fuglede's Theorem). $$[N(p_j)^* N(p_i)]=[N(p_i) N(p_j)^*]=\begin{bmatrix} N(p_1) \\ N(p_2) \\ \vdots \\ N(p_n) \end{bmatrix} \begin{bmatrix} N(p_1)^* & N(p_2)^* & \cdots N(p_n)^* \end{bmatrix}\geq 0.$$
$N(p_i)$ extends $T(p_i)$ and therefore $P_\mathcal{H} N(p_j)^* N(p_i) |_\mathcal{H} = T(p_j)^* T(p_i)$. By projecting on $\mathcal{H}^n$, we get the desired inequality.

For the third condition:
\begin{align*}
 &[\tilde{T}(s_i^* t^* t s_j)]\\
=&[\tilde{T}(q_i p_j ab, ab p_i q_j)] \\
=&[T(ab)^* T(q_i)^* T(p_j)^* T(p_i) T(q_j) T(ab)] \\
=&\diag(T(q_1)^*, T(q_2)^*, \cdots, T(q_n)^*) [T(ab)^* T(p_j)^* T(p_i) T(ab)] \diag(T(q_1),\cdots,T(q_n))
\end{align*} 
Therefore, it suffices to show (with $C_t=1$ in the condition)
$$[T(ab)^* T(p_j)^* T(p_i) T(ab)]\leq [T(p_j)^* T(p_i)]$$
Similar to the previous case, it suffices to show $$[N(ab)^* N(p_j)^* N(p_i) N(ab)]\leq [N(p_j)^* N(p_i)]$$
Let $X=[N(p_j)^* N(p_i)]\geq 0$ and $D=\diag(N(ab),\cdots,N(ab))$. Since $D$ and $X$ $*$-commute, and thus $D$ and  $X^{1/2}$ also $*$-commute. We have 
$$ D^* X D=X^{1/2} D^*D X^{1/2} \leq \|N(ab)\| X.$$

Since $N$ is contractive, this shows $D^* X D\leq X$. Therefore, all conditions in Theorem \ref{thm.principal} are met, and thus there exists a contractive representation $S:Q\to\bh{L}$ such that $\tilde{T}(p,q)=P_\mathcal{H} S(p,q) |_\mathcal{H}$. Denote $M(p)=S(0,p)$. Then $M:P\to\bh{L}$ is a representation of $P$, and moreover, $$T(p)^* T(p) = P_\mathcal{H} S(p,p) |_\mathcal{H}=P_\mathcal{H} M(p)^* M(p) |_\mathcal{H}.$$
By Lemma \ref{lm.1}, we know $M(p)$ extends $T(p)$ and therefore $M$ is a normal extension. 
\end{proof}
\begin{remark} When the semigroup is $P=\mathbb{N}^k$, Theorem \ref{thm.normalmap} is trivial: for a normal map $N:\mathbb{N}^k\to\bh{K}$, one may define a normal representation by sending each generator $e_i$ to $N(e_i)$. However, it is not clear how we can derive a normal representation from a normal map when the semigroup does not have nice generators. For example, we have seen this issue in Example \ref{ex.1} where the semigroup $P=\mathbb{N}\backslash\{1\}$ is finitely generated. This result shows that finding a commuting family of normal extensions for $\{T(p)\}_{p\in P}$ is equivalent of finding a normal representation that extends $T$. 
\end{remark} 
\begin{corollary}\label{cor.semigp} Let $P$ be a commutative unital semigroup generated by $\{p_i\}_{i\in I}$, and $T:P\to\bh{H}$ a unital contractive representation. Then the family $\{T(p_i)\}_{i\in I}$ has commuting normal extensions $\{N_i\}_{i\in I}$ if and only if there exists a normal representation $N:P\to\bh{K}$ such that each $N(p)$ extends $T(p)$.  
\end{corollary} 
\begin{proof} The backward direction is obvious. Now assuming $\{T(p_i)\}_{i\in I}$ has commuting normal extension $\{N_i\}_{i\in I}$. For each element $p\in P$, write $p$ as a finite product of $\{p_i\}_{i\in I}$ and define $N(p)$ to be the corresponding product of $T(p_i)$. Since $N_i$ commutes with one another, we obtain a normal map $\overline{N}:P\to\bh{L}$ where $\{\overline{N}(p)\}_{p\in P}$ is a family of commuting normal operators where $\overline{N}(p)$ extends $T(p)$. Theorem \ref{thm.normalmap} implies the existence of the desired normal representation $N$.
\end{proof}

\begin{remark} Corollary \ref{cor.semigp} shows that for a contractive representation $T:P\to\bh{H}$, it suffices to extends the image of $T$ on a set of generators. Since Athavale's result still holds for an infinite family of operators (Corollary \ref{cor.inf}), we may use Condition (\ref{eq.athavale}) to check if the set of generators have a commuting normal extension. However, when the semigroup has too many generators, Condition (\ref{eq.athavale}) is hard to check. We shall give another equivalent condition for an abelian lattice ordered group in the next section. 
\end{remark} 

\section{Normal Extensions For Lattice Ordered Semigroups}

A lattice ordered semigroup $P$ is a unital normal semigroup inside a group $G=P^{-1}P$ that induces a lattice order. Given a unital normal semigroup $P\subseteq G=P^{-1}P$, there is a natural partial order on $G$ given by $x\leq y$ when $x^{-1} y\in P$. If any two elements $g,h$ in $G$ has a least upper bound (also called the join $g\vee h$) and greatest lower bound (also called the meet $g\wedge h$) under this partial order, the partial order is called a lattice order.
\begin{example} (Examples of Lattice Ordered Semigroups)
\begin{enumerate}[(i)]
\item $\mathbb{N}^k$ is a lattice ordered semigroup inside $\mathbb{Z}^k$ for any $k\in\mathbb{N}\bigcup\{\infty\}$. 
\item $\mathbb{R}^+$ is a lattice ordered semigroup inside $\mathbb{R}$. Notice that $\mathbb{R}^+$ is not countably generated.
\item More generally, if the partial order induced by $P$ is a total order, or equivalently, $G=P\bigcup P^{-1}$, then $P$ is also a lattice ordered semigroup in $G$. 
\item If $P_i$ are lattice ordered semigroups inside $G_i$, then their product $\prod P_i$ is also a lattice ordered semigroup inside $\prod G_i$.
\item If $X$ is a topological space and $C^+(X)$ contains all the non-negative continuous function on $X$. Then $C^+(X)$ is a lattice ordered semigroup inside $C(X)$, where the group operation is point-wise addition. 
\item Even though our focus is on abelian lattice ordered semigroups, there are non-abelian lattice ordered semigroups. Consider an uncountable totally ordered set $X$, and define $G$ to be the set of all order preserving bijections on $X$. $G$ is a group under composition. Define $P=\{\alpha\in G: \alpha(x)\geq x\}$, then $P$ is a non-abelian lattice ordered semigroup in $G$\cite{LatticeOrderBookIntro}.
\end{enumerate}
\end{example} 
If $P$ is a lattice ordered semigroup inside $G$, then every element $g\in G$ has a unique positive and negative part $g_+, g_-$, in the sense that $g=g_-^{-1} g_+ $ and $g_+\wedge g_-=e$. This notion of positive and negative part is essential in defining a regular dilation.
For a lattice ordered semigroup $P$ inside $G$, a representation $T:P\to\bh{H}$ has a dilation $U:G\to\bh{K}$ if $U$ is a unitary representation of $G$ on a Hilbert space $\mathcal{K}\supseteq\mathcal{H}$ such that for any $p\in P$, $$T(p)=P_\mathcal{H} U(p)|_\mathcal{H}.$$
Such a dilation is called regular if for any $g\in G$, $$T(g_-)^* T(g_+)=P_\mathcal{H} U(g)|_\mathcal{H}.$$
There is a dual version of regular dilation that call such a dilation $*$-regular if for any $g\in G$, $$T(g_+)T(g_-)^*=P_\mathcal{H} U(g)|_\mathcal{H}.$$

These two definitions are equivalent in the sense that $T$ is $*$-regular if and only if $T^*:P^{-1}\to\bh{H}$ where $T^*(p^{-1})=T(p)^*$ is regular \cite[Proposition 2.5]{Boyu2014}. We call a representation $T$ regular if it has a regular dilation. 

A well known result due to Sarason shows that such Hilbert space $\mathcal{K}$ can be decomposed as $\mathcal{K}=\mathcal{K}_+\oplus\mathcal{H}\oplus\mathcal{K}_-$ where under such such decomposition, $$U(p)=\begin{bmatrix} * & 0 & 0 \\ * & T(p) & 0 \\ * & * & * \end{bmatrix}.$$
Regular dilations were first studied by Brehmer \cite{Brehmer1961} where he gave a necessary and sufficient condition for a representation on $\mathbb{N}^\Omega$ to be regular:
\begin{theorem}[Brehmer]\label{thm.Brehmer} Let $\Omega$ be a set, and denote $\mathbb{Z}^\Omega$ to be the set of $(t_\omega)_{\omega\in\Omega}$ where $t_\omega\in\mathbb{Z}$ and $t_\omega=0$ except for finitely many $\omega$. Also, for a finite set $V\subset\Omega$, denote $e_V\in\mathbb{Z}^\Omega$ to be $1$ at those $\omega\in V$ and $0$ elsewhere. If $\{T_\omega\}_{\omega\in\Omega}$ is a family of commuting contractions, we may define a contractive representation $T:\mathbb{Z}_+^\Omega\to\bh{H}$ by $$T(t_\omega)_{\omega\in\Omega}=\prod_{\omega\in\Omega} T_\omega^{t_\omega}.$$
Then $T$ is regular if and only if for any finite $U\subseteq\Omega$, the operator $$\sum_{V\subseteq U} (-1)^{|V|} T(e_V)^* T(e_V)\geq 0.$$
\end{theorem}
Recently, the author extended this result to an arbitrary lattice ordered semigroup (not necessarily abelian) \cite{Boyu2014}:
\begin{theorem}\label{thm.regular} Let $P$ be a lattice ordered semigroup in $G$ and $T:P\to\bh{H}$ be a contractive representation. Define $\tilde{T}:G\to\bh{H}$ by $\tilde{T}(g)=T(g_-)^* T(g_+)$. Then $T$ is regular if and only if for any $p_1,\cdots,p_n\in P$ and $g\in P$ where $g\wedge p_i=e$ for all $i=1,2,\cdots,n$, we have
$$\left[T(g)^* \tilde{T}(p_i p_j^{-1})T(g)\right]\leq \left[\tilde{T}(p_i p_j^{-1})\right].$$
\end{theorem}
Although it is observed that Condition (\ref{eq.athavale}) implies a representation $T:\mathbb{N}^m\to\bh{H}$ has regular dilation \cite{A1992}, the converse is not true. However, we shall prove that Athavale's result is equivalent of saying that a certain representation $T^\infty$ is regular. First of all, define $\mathbb{N}^{m\times \infty}$ by taking the product of infinitely many copies of $\mathbb{N}^m$, in other words, $\mathbb{N}^{m\times \infty}$ is the abelian semigroup generated by $(e_{i,j})_{\substack{1\leq i\leq m \\ j\in\mathbb{N}}}$. Consider $T^\infty:\mathbb{N}^{m\times \infty}\to\bh{H}$ where $T^\infty$ sends each generator $e_{i,j}$ to $T_i$. 
\begin{lemma}\label{lm.motive} As defined above, $T^\infty$ is regular if and only if $T$ satisfies the Condition (\ref{eq.athavale}). 
\end{lemma}
\begin{proof} It suffices to verify Condition (\ref{eq.athavale}) is equivalent to Brehmer's condition on $\mathbb{N}^{m\times\infty}$ in Theorem \ref{thm.Brehmer}. For any finite set $U\subseteq \{1,2,\cdots,m\}\times \mathbb{N}$, denote by $n_i$ the number of $u\in U$ whose first coordinate is $i$. For any subset $V\subseteq U$, denote by $k_i$ the number of $v\in V$ whose first coordinate is $i$. It is clear that $0\leq k_i\leq n_i$. Notice that $T(e_V)=T_1^{k_1} T_2^{k_2}\cdots T_m^{k_m}$, and among all subsets of $U$, there are exactly ${n_1\choose k_1}\cdots{n_m\choose k_m}$ subsets $V$ that have $k_i$ elements whose first coordinate is $i$. Therefore, 
\begin{align*}
& \sum_{V\subseteq U} (-1)^{|V|} T(e_V)^* T(e_V) \\ 
=& \sum_{0\leq k_i\leq n_i} (-1)^{k_1+k_2+\cdots+k_m} {n_1\choose k_1}\cdots{n_m\choose k_m} T_1^{*k_1}T_2^{*k_2}\cdots T_m^{*k_m}T_m^{k_m}\cdots T_1^{k_1}.
\end{align*}
Hence, Brehmer's condition holds if and only if $T$ satisfies Condition (\ref{eq.athavale}).
\end{proof}
Notice that Condition (\ref{eq.athavale}) cannot be generalized directly to arbitrary abelian lattice ordered semigroups when the semigroup lacks generators. However, Lemma \ref{lm.motive} motivates us to consider $T^\infty$ in an abelian lattice ordered semigroup: for a lattice ordered semigroup $P$ inside a group $G$, define $P^\infty=\prod_{i=1}^\infty P$ to be the abelian semigroup generated by infinitely many identical copies of $P$. We shall denote $p\otimes \delta_n$ to be $p$ inside the $n$-th copy of $P^\infty$. A typical element of $P^\infty$ can be denoted by $\sum_{i=1}^N p_i\otimes \delta_i$ for some large enough $N$. $P^\infty$ is naturally a lattice ordered semigroup inside the group $G^\infty$, where $$\left(\sum_{i=1}^N p_i\otimes \delta_i\right) \wedge \left(\sum_{i=1}^N q_i\otimes \delta_i\right)=\sum_{i=1}^N p_i\wedge q_i \otimes \delta_i.$$
Our main result shows that $T^\infty$ being regular is equivalent to having a normal extension. 
\begin{theorem}\label{thm.main} Let $T:P\to\bh{H}$ be a contractive representation on an abelian lattice ordered semigroup. Define $T^{\infty}:P^\mathbb{N}\to\bh{H}$ by $T^\infty(p,n)=T(p)$ for any $n$. Then the following are equivalent:
\begin{enumerate}[(i)]
\item $T$ has a contractive normal extension to a representation $N:P\to\bh{K}$. In other words, there exists a contractive normal representation $N:P\to\bh{K}$ such that for all $p\in P$, $T(p)=N(p)|_\mathcal{H}$. 
\item $T^\infty$ is regular.
\end{enumerate}
\end{theorem}
\begin{proof}
\textit{(i)}$\Rightarrow$\textit{(ii)}: First of all notice that the family $\{N(p)\}_{p\in P}$ $*$-commutes due to Fuglede's theorem. Define $N^\infty$ by sending $N^\infty(p,n)=N(p)$ for all $p\in P, n\in\mathbb{N}$. Then for any $s,t\in P^\infty$, $N^\infty(s),N^\infty(t)$ are a finite product of operators in $\{N(p)\}_{p\in P}$ and therefore they also $*$-commute. In particular, $N^\infty$ is Nica-covariant and therefore is regular \cite[Theorem 4.1]{Boyu2014}. Since $N$ extends $T$, $N^\infty$ also extends $T^\infty$, and therefore for any $s,t\in P^\infty$, $$P_\mathcal{H} N^\infty(t)^* N^\infty(s)|_\mathcal{H}=T^\infty(t)^* T^\infty(s).$$
$N^\infty$ satisfies the condition in Theorem \ref{thm.regular}, and by projecting onto $\mathcal{H}$, $T^\infty$ also satisfies this condition and thus is regular.

\textit{(ii)}$\Rightarrow$\textit{(i)}: Let $U:G^\infty\to\bh{K}$ be a regular unitary dilation of $T^\infty$, and decompose $\mathcal{K}=\mathcal{K}_+\oplus\mathcal{H}\oplus\mathcal{K}_-$ so that under such decomposition, for each $w\in P^\infty$, $$U(w)=\begin{bmatrix} * & 0 & 0 \\ * & T(w) & 0 \\ * & * & * \end{bmatrix}.$$
Fix $p\in P$, denote $U_i(p)=U(p\otimes\delta_i)$. Under the decomposition $\mathcal{K}=\mathcal{K}_+\oplus\mathcal{H}\oplus\mathcal{K}_-$, let $$U_i(p)=\begin{bmatrix} A_i & 0 &0 \\ B_i & T(p) & 0 \\ C_i & D_i & E_i \end{bmatrix}.$$
First by regularity of $U$, for any $i\neq j$, 
\begin{align*}
T(p)^* T(p) &= P_\mathcal{H} U(p\otimes \delta_i - p\otimes \delta_j)|_\mathcal{H} \\
&= P_\mathcal{H} U_j(p)^* U_i(p)|_\mathcal{H} \\
&= \left. P_\mathcal{H} \begin{bmatrix} A_j^* & B_j^* & C_j^* \\ 0 & T(p)^* & D_j^* \\ 0& 0 & E_j^*\end{bmatrix} \begin{bmatrix} A_i & 0 &0 \\ B_i & T(p) & 0 \\ C_i & D_i & E_i \end{bmatrix}\right|_\mathcal{H} \\
&= \left. P_\mathcal{H} \begin{bmatrix} * & * & * \\ * & T(p)^*T(p)+D_j^* D_i & * \\ *& *& *\end{bmatrix}\right|_\mathcal{H} \\
\end{align*} 
Therefore, each $D_j^* D_i=0$ whenever $i\neq j$. When $i=j$, since $U$ is a unitary representation, $U_i(p)$ is a unitary, and thus $D_i^* D_i=I-T(p)^* T(p)$. Now fix $\epsilon>0$, denote
$$\Lambda_\epsilon=\{\lambda=(\lambda_i)_{i=1}^\infty\in c_{00}: \sum_{i=1}^\infty \lambda_i=1, 0\leq\lambda_i\leq 1, \|\lambda\|_2<\epsilon\}.$$
This set is non-empty since we may let $\lambda_i=\frac{1}{n}$ for $1\leq i\leq n$, and $0$ otherwise. This gives $\|\lambda\|_2=\frac{1}{\sqrt{n}}$, which can be arbitrarily small as $n\to\infty$. For each $\lambda\in\Lambda_\epsilon$, denote $N_\lambda=\sum_{i=1}^\infty \lambda_i U_i(p)$, which converges since $\lambda$ has finite support. Denote $$\mathcal{N}_\epsilon=\{N_\lambda:\lambda\in\Lambda_\epsilon\}$$
Notice that $P_\mathcal{H} N_\lambda|_\mathcal{H} = \sum_{i=1}^\infty \lambda_i T(p) = T(p)$. Therefore, under the decomposition $\mathcal{K}=\mathcal{K}_+\oplus\mathcal{H}\oplus\mathcal{K}_-$, $$N_\lambda=\begin{bmatrix} A_\lambda & 0& 0 \\ B_\lambda & T(p) & 0 \\ C_\lambda & D_\lambda & E_\lambda \end{bmatrix}.$$
Here, $D_\lambda=\sum_{i=1}^\infty \lambda_i D_i$ and thus
\begin{align*}
D_\lambda^* D_\lambda =& \sum_{i,j=1}^\infty \overline{\lambda_i}\lambda_j D_i^* D_j \\
=& \sum_{i=1}^\infty |\lambda_i|^2 D_i^* D_i \\
\end{align*}
Here we used the fact that $D_i^* D_j=0$ whenever $i\neq j$. Note that each $D_i^* D_i=I-T(p)^*T(p)$, which is contractive. Hence, $$\|D_\lambda^* D_\lambda\|\leq \|\lambda\|_2^2<\epsilon^2$$

Each $N_\lambda$ is a convex combination of $U_i$ and thus is contained in the convex hull of $U_i$, which is also contained in the unit ball in $\bh{K}$. Observe that each $\mathcal{N}_\epsilon$ is also convex. Therefore, the convexity implies their SOT${}^*$ and WOT closures agree (here, $SOT^*-lim T_n=T$ if $T_n$ and $T_n^*$ converges to $T$ and $T^*$ respectively in SOT.). Hence, 
$$\overline{\mathcal{N}_\epsilon}^{SOT^*}=\overline{\mathcal{N}_\epsilon}^{WOT}\subseteq \overline{\operatorname{conv}}^{WOT}\{U_i\}\subseteq b_1(\bh{K})$$
The Banach Alaoglu theorem gives $b_1(\bh{K})$ is WOT-compact, and therefore $\overline{\mathcal{N}_\epsilon}^{WOT}$ is a decreasing nest of WOT-compact sets. By the Cantor intersection theorem, $$\bigcap_{\epsilon>0} \overline{\mathcal{N}_\epsilon}^{SOT^*}=\bigcap_{\epsilon>0} \overline{\mathcal{N}_\epsilon}^{WOT}\neq \emptyset$$
Pick $N(p) \in \bigcap_{\epsilon>0} \overline{\mathcal{N}_\epsilon}^{SOT^*}$. Then for any $\epsilon>0$, we can choose a net $(N_\lambda)_{\lambda\in I_\epsilon}$, where $I_\epsilon\subseteq\Lambda_\epsilon$, such that $SOT^*-\lim_{I_\epsilon} N_\lambda = N(p) $ and thus $SOT^*-\lim_{I_\epsilon} N_\lambda^*=N(p)^*$. Now both $N_\lambda,N_\lambda^*$ are uniformly bounded by 1 since they are all contractions. Hence, their product is SOT-continuous.
\begin{align*}
SOT-\lim_\Lambda N_\lambda^* N_\lambda &= N(p)^* N(p) \\
SOT-\lim_\Lambda  N_\lambda N_\lambda^* &= N(p) N(p)^* 
\end{align*}
But since $U_i$ are commuting unitaries and thus $*$-commute, $N_\lambda$ is normal. Hence, $N(p)^* N(p)=N(p)N(p)^*$ and $N(p)$ is normal. 

Consider $N(p)\in\bh{K}$ under the decomposition $\mathcal{K}=\mathcal{K}_+\oplus\mathcal{H}\oplus\mathcal{K}_-$, each entry must be the WOT-limit of $(N_\lambda)_{\lambda\in I_\epsilon}$ and therefore it has the form $$N(p)=\begin{bmatrix} A(p) & 0 & 0 \\ B(p) & T(p) & 0 \\ C(p) & D(p) & E(p) \end{bmatrix}.$$
Since $(D_\lambda)_{\lambda\in I_\epsilon}$ WOT-converges to $D(p)$, and for each $\lambda\in\Lambda_\epsilon$, $\|D_\lambda\|<\epsilon$. Therefore, $\|D(p)\|<\epsilon$ for every $\epsilon>0$ and thus $D(p)=0$. Hence $\mathcal{H}$ is invariant for $N(p)$, whence $N(p)$ is a normal extension for $T(p)$. 

The procedure above gives a normal map $N:P\to\bh{K}$ where each $N(p)$ is a normal contraction that extends $T(p)$. Notice $N(p)$ is a WOT-limit of convex combinations of $\{U_i(p)\}_{i\in\mathbb{N}}$, where the family $\{U_i(p)\}_{i,p}$ is commuting since $P$ is abelian. Any convex combination of $\{U_i(p)\}_{i\in\mathbb{N}}$ also commutes with any convex combination of $\{U_i(q)\}_{i\in\mathbb{N}}$. Therefore, $\{N(p)\}_{p\in P}$ is also a commuting family of normal operators. By Theorem \ref{thm.normalmap}, there exists a normal representation $N:P\to\bh{L}$ that extends $T$. 
\end{proof} 

As an immediate corollary, Theorem \ref{thm.Athavale} can be extended to any family of commuting contractions $\{T(\omega)\}_{\omega\in\Omega}$ by considering Brehmer's condition on $\mathbb{N}^{\Omega\times\infty}$. 

\begin{corollary}\label{cor.inf} Let $\{T_i\}_{i\in I}$ be a family of commuting contractions. Then there exists a family of commuting normal contractions $\{N_i\}_{i\in I}$ that extends $\{T_i\}_{i\in I}$ if and only if for any finite set $F\subseteq I$, $\{T_i\}_{i\in F}$ satisfies Condition (\ref{eq.athavale}). 
\end{corollary} 

It is known that isometric representations of lattice ordered semigroups are automatically regular \cite[Corollary 3.8]{Boyu2014}. Therefore, if $T:P\to\bh{H}$ is an isometric representation, then $T^\infty:P^\infty\to\bh{H}$ is also an isometric representation and thus $T$ has a subnormal extension.

\begin{corollary} Every isometric representation of an abelian lattice ordered semigroup has a contractive subnormal extension. 
\end{corollary}

\bibliographystyle{plain}
\bibliography{Subnormality}

\end{document}